\documentclass[a4paper,12pt,oneside]{article}
\usepackage{amsmath,amsfonts,amssymb,amsthm}

\theoremstyle{plain}
\newtheorem{theorem}{Theorem}[section]

\newtheorem{lemma}[theorem]{Lemma}
\newtheorem{proposition}[theorem]{Proposition}

\theoremstyle{definition}
\newtheorem{definition}[theorem]{Definition}

\newtheorem{example}[theorem]{Example}
\newtheorem{remark}[theorem]{Remark}

\numberwithin{equation}{section}

\newcommand{\R}{{\mathbb R}}
\newcommand{\N}{{\mathbb N}}

\providecommand{\vint}[1]{\mathchoice
          {\mathop{\vrule width 5pt height 3 pt depth -2.5pt
                  \kern -9pt \kern 1pt\intop}\nolimits_{\kern -5pt{#1}}}
          {\mathop{\vrule width 5pt height 3 pt depth -2.6pt
                  \kern -6pt \intop}\nolimits_{\kern -3pt{#1}}}
          {\mathop{\vrule width 5pt height 3 pt depth -2.6pt
                  \kern -6pt \intop}\nolimits_{\kern -3pt{#1}}}
          {\mathop{\vrule width 5pt height 3 pt depth -2.6pt
                  \kern -6pt \intop}\nolimits_{\kern -3pt{#1}}}}

\newcommand{\eps}{\varepsilon}
\newcommand{\loc}{{\mbox{\scriptsize{loc}}}}

\newcommand{\BV}{\mathrm{BV}}
\newcommand{\liploc}{\mathrm{Lip}_{\mathrm{loc}}}

\DeclareMathOperator{\dist}{dist}
\DeclareMathOperator{\diam}{diam}

\DeclareMathOperator{\Lip}{Lip}

\DeclareMathOperator{\supp}{supp}

\makeatletter
\def\blfootnote{\xdef\@thefnmark{}\@footnotetext}
\makeatother

\begin{document}
\title{Extensions and traces of functions of bounded variation on metric spaces
}
\author{Panu Lahti}
\maketitle
\begin{abstract}
In the setting of a metric space equipped with a doubling measure and supporting a Poincar\'e inequality, and based on results by Bj\"orn and Shanmugalingam (2007, \cite{BS}), we show that functions of bounded variation can be extended from any bounded uniform domain to the whole space. Closely related to extensions is the concept of boundary traces, which have previously been studied by Hakkarainen et al. (2014, \cite{HKLL}). On spaces that satisfy a suitable locality condition for sets of finite perimeter, we establish some basic results for the traces of functions of bounded variation. Our analysis of traces also produces novel results on the behavior of functions of bounded variation in their jump sets.
\end{abstract}

\blfootnote{
2010 Mathematics Subject Classification. 30L99, 26B30, 46E35.

{\it Keywords\,}: boundary trace, bounded variation, extension, jump set, locality, uniform domain.
}

\section{Introduction}

A classical Euclidean result on the extension of Sobolev functions and functions of bounded variation, abbreviated as $\BV$ functions, is that any bounded domain with a Lipschitz boundary allows such extensions, see e.g. \cite[Proposition 3.21]{AFP}. For Sobolev functions, this result was generalized to so-called $(\eps,\delta)$-domains by Jones \cite{J}. On metric spaces, extension results for various classes of functions, such as Haj\l{}asz-Sobolev functions and H\"older continuous functions, have been derived in e.g. \cite{HKT} and \cite{BBS}. One result on the extension of $\BV$ functions on metric spaces is given by Baldi and Montefalcone \cite{BM} who show, in essence, that if sets of finite perimeter can be extended from a domain, then so can general $\BV$ functions. However, a simple geometric condition ensuring the extendability of $\BV$ functions appears to be missing.

Bj\"orn and Shanmugalingam show in \cite{BS} that every \emph{uniform domain} $\Omega$ is an extension domain for Newton-Sobolev functions $N^{1,p}(\Omega)$, with $p\ge 1$. On the other hand, $\BV$ functions on metric spaces are defined by relaxation with Newton-Sobolev functions, see \cite{Amb} and \cite{Mir}. Thus the extension result of \cite{BS} can be applied to the class $\BV$ in a fairly straightforward manner, as presented in this note.

On the other hand, the concepts of extensions and \emph{boundary traces} are closely related. Classical treatments of boundary traces of $\BV$ functions can be found in e.g. \cite[Chapter 3]{AFP} and \cite[Chapter 2]{Gi}, and a standard assumption is again a Lipschitz boundary. On the other hand, in the metric setting, results on boundary traces seem to be largely absent, with the exception of \cite{HKLL}, where boundary traces of $\BV$ functions were defined on the boundaries of certain $\BV$ extension domains.

In this note, we present a different approach to traces, where more is assumed of the space but less of the domain. More precisely, we assume a certain locality condition that essentially states that any two sets of finite perimeter ``look the same'' near almost every point in which their measure theoretic boundaries intersect. Then we can prove the existence of \emph{interior traces} of $\BV$ functions on the measure theoretic boundary of any set of finite perimeter, and also prove the existence of boundary traces on the measure theoretic boundary of any extension domain.

In \cite{KKST}, pointwise properties of $\BV$ functions on metric spaces were studied, and in particular a Lebesgue point theorem for $\BV$ functions outside their jump sets was given. Since the super-level sets of a $\BV$ function are sets of finite perimeter, we are able to apply our analysis of traces to prove novel results on the behavior of a $\BV$ function in its jump set, extending classical results to metric spaces and strengthening results found in \cite{KKST}.

\section{Preliminaries}

In this section we introduce the necessary definitions and assumptions.

In this paper, $(X,d,\mu)$ is a complete metric space equipped
with a Borel regular outer measure $\mu$.
The measure is assumed to be doubling, meaning that there exists a constant $c_d>0$ such that
\[
0<\mu(B(x,2r))\leq c_d\mu(B(x,r))<\infty
\]
for every ball $B=B(x,r)$ with center $x\in X$ and radius $r>0$. This implies that
\begin{equation}\label{eq:definition of Q}
\frac{\mu(B(y,r))}{\mu(B(x,R))}\ge C\left(\frac{r}{R}\right)^Q
\end{equation}
for every $0<r\le R$ and $y\in B(x,R)$, and some $Q>1$ and $C>0$ that only depend on $c_d$.
In general, $C$ will denote a positive constant whose value is not
necessarily the same at each occurrence.

We recall that a complete metric space endowed with a doubling measure is proper,
that is, closed and bounded sets are compact. Since $X$ is proper, for any open set $\Omega\subset X$
we define e.g. $\textrm{Lip}_{\loc}(\Omega)$ as the space of
functions that are Lipschitz in every $\Omega'\Subset\Omega$.
Here $\Omega'\Subset\Omega$ means that $\Omega'$ is open and that $\overline{\Omega'}$ is a
compact subset of $\Omega$.


For any set $A\subset X$ and $0<R<\infty$, the restricted spherical Hausdorff content
of codimension $1$ is defined as
\[
\mathcal{H}_{R}(A):=\inf\left\{ \sum_{i=1}^{\infty}\frac{\mu(B(x_{i},r_{i}))}{r_{i}}:\,A\subset\bigcup_{i=1}^{\infty}B(x_{i},r_{i}),\,r_{i}\le R\right\}.
\]
The Hausdorff measure of codimension $1$ of a set
$A\subset X$ is
\[
\mathcal{H}(A):=\lim_{R\rightarrow0}\mathcal{H}_{R}(A).
\]

The (topological) boundary $\partial E$ of a set $E\subset X$ is
defined as usual. The measure theoretic boundary $\partial^{*}E$ is defined as the set of points $x\in X$
in which both $E$ and its complement have positive upper density, i.e.
\[
\limsup_{r\rightarrow0}\frac{\mu(B(x,r)\cap E)}{\mu(B(x,r))}>0\quad\;\textrm{and}\quad\;\limsup_{r\rightarrow0}\frac{\mu(B(x,r)\setminus E)}{\mu(B(x,r))}>0.
\]

A curve is a rectifiable continuous mapping from a compact interval
to $X$, and is usually denoted by the symbol $\gamma$.
The length of a curve $\gamma$
is denoted by $\ell_{\gamma}$. We will assume every curve to be parametrized
by arc-length, which can always be done (see e.g. \cite[Theorem 3.2]{Hj}).

A nonnegative Borel function $g$ on $X$ is an upper gradient 
of an extended real-valued function $u$
on $X$ if for all curves $\gamma$ on $X$, we have
\[
|u(x)-u(y)|\le \int_\gamma g\,ds
\]
whenever both $u(x)$ and $u(y)$ are finite, and 
$\int_\gamma g\, ds=\infty $ otherwise.
Here $x$ and $y$ are the end points of $\gamma$.

We consider the following norm
\[
\Vert u\Vert_{N^{1,1}(X)}:=\Vert u\Vert_{L^1(X)}+\inf_g\Vert g\Vert_{L^1(X)},
\]
with the infimum taken over all upper gradients $g$ of $u$. 
The Newton-Sobolev, or Newtonian space is defined as
\[
N^{1,1}(X):=\{u:\|u\|_{N^{1,1}(X)}<\infty\}/{\sim},
\]
where the equivalence relation $\sim$ is given by $u\sim v$ if and only if 
\[
\Vert u-v\Vert_{N^{1,1}(X)}=0.
\]
Similarly, we can define $N^{1,1}(\Omega)$ for an open set $\Omega\subset X$. For more on Newtonian spaces, we refer to \cite{BB}.

Next we recall the definition and basic properties of functions
of bounded variation on metric spaces, see \cite{Mir}. 
For $u\in L^1_{\text{loc}}(X)$, we define the total variation of $u$ as
\[
\|Du\|(X):=\inf\Big\{\liminf_{i\to\infty}\int_X g_{u_i}\,d\mu:\, u_i\in \Lip_{\loc}(X),\, u_i\to u\textrm{ in } L^1_{\text{loc}}(X)\Big\},
\]
where $g_{u_i}$ is an upper gradient of $u_i$.
We say that a function $u\in L^1(X)$ is of bounded variation, 
and denote $u\in\BV(X)$, if $\|Du\|(X)<\infty$. 
Moreover, a $\mu$-measurable set $E\subset X$ is said to be of finite perimeter if $\|D\chi_E\|(X)<\infty$.
By replacing $X$ with an open set $\Omega\subset X$ in the definition of the total variation, we can define $\|Du\|(\Omega)$.
The $\BV$ norm is given by
\[
\Vert u\Vert_{\BV(\Omega)}:=\Vert u\Vert_{L^1(\Omega)}+\Vert Du\Vert(\Omega).
\]
For an arbitrary set $A\subset X$, we define
\[
\|Du\|(A):=\inf\bigl\{\|Du\|(\Omega):\,\Omega\supset A,\,\Omega\subset X
\text{ is open}\bigr\}.
\]
If $u\in\BV(\Omega)$, $\|Du\|(\cdot)$ Radon measure of finite mass on $\Omega$ \cite[Theorem 3.4]{Mir}.

We also denote the perimeter of $E$ in $\Omega$ by
\[
P(E,\Omega):=\|D\chi_E\|(\Omega).
\]
We have the following coarea formula given by Miranda in \cite[Proposition 4.2]{Mir}: if $F\subset X$ is a Borel set and $u\in \BV(X)$, we have
\begin{equation}\label{eq:coarea}
\|Du\|(F)=\int_{-\infty}^{\infty}P(\{u>t\},F)\,dt.
\end{equation}

We always assume that $X$ supports a $(1,1)$-Poincar\'e inequality,
meaning that for some constants $c_P>0$ and $\lambda \ge 1$, for every
ball $B(x,r)$, for every locally integrable function $u$,
and for every upper gradient $g$ of $u$, we have 
\[
\vint{B(x,r)}|u-u_{B(x,r)}|\, d\mu 
\le c_P r\,\vint{B(x,\lambda r)}g\,d\mu,
\]
where 
\[
u_{B(x,r)}:=\vint{B(x,r)}u\,d\mu :=\frac 1{\mu(B(x,r))}\int_{B(x,r)}u\,d\mu.
\]
The $(1,1)$-Poincar\'e inequality implies the so-called Sobolev-Poincar\'e inequality, see e.g. \cite[Theorem 4.21]{BB}, and by approximation with get the following Sobolev-Poincar\'e inequality for $\BV$ functions. There exists $C>0$, depending only on the doubling constant and the constants in the Poincar\'e inequality, such that for every ball $B(x,r)$ and every $u\in L^1_{\loc}(X)$, we have
\[
\left(\,\vint{B(x,r)}|u-u_{B(x,r)}|^{Q/(Q-1)}\,d\mu\right)^{(Q-1)/Q}
\le Cr\frac{\Vert Du\Vert (B(x,2\lambda r))}{\mu(B(x,2\lambda r))}.
\]
Recall the definition of the number $Q$ from \eqref{eq:definition of Q}. Moreover, for functions $u\in L^1_{\loc}(X)$ with approximate limit $0$ at $x$, i.e. $u^{\wedge}(x)=u^{\vee}(x)=0$, we have
\begin{equation}\label{eq:asymptotic poincare}
\limsup_{r\to 0}\left(\,\vint{B(x,r)}|u|^{Q/(Q-1)}\,d\mu\right)^{(Q-1)/Q}
\le C\limsup_{r\to 0}r\frac{\Vert Du\Vert (B(x,2\lambda r))}{\mu(B(x,2\lambda r))},
\end{equation}
see \cite[Lemma 3.1]{KKST}.

Given a set of locally finite perimeter $E\subset X$, for $\mathcal H$-a.e. $x\in \partial^*E$ we have
\begin{equation}\label{eq:density of E}
\gamma \le \liminf_{r\to 0} \frac{\mu(E\cap B(x,r))}{\mu(B(x,r))} \le \limsup_{r\to 0} \frac{\mu(E\cap B(x,r))}{\mu(B(x,r))}\le 1-\gamma,
\end{equation}
where $\gamma \in (0,1/2]$ only depends on the doubling constant and the constants in the Poincar\'e inequality \cite[Theorem 5.4]{Amb}.
For a Borel set $F\subset X$ and a set of finite perimeter $E\subset X$, we know that
\begin{equation}\label{eq:def of theta}
\Vert D\chi_{E}\Vert(F)=\int_{\partial^{*}E\cap F}\theta_E\,d\mathcal H,
\end{equation}
where $\theta_E:X\mapsto [\alpha,c_d]$, with $\alpha=\alpha(c_d,c_P,\lambda)>0$, see \cite[Theorem 5.3]{Amb} and \cite[Theorem 4.6]{AMP}.

The jump set of $u\in\BV(X)$ is defined as
\[
S_{u}:=\{x\in X:\, u^{\wedge}(x)<u^{\vee}(x)\},
\]
where $u^{\wedge}$ and $u^{\vee}$ are the lower and upper approximate limits of $u$ defined as
\[
u^{\wedge}(x):
=\sup\left\{t\in\overline\R:\,\lim_{r\to0}\frac{\mu(B(x,r)\cap\{u<t\})}{\mu(B(x,r))}=0\right\}
\]
and
\[
u^{\vee}(x):
=\inf\left\{t\in\overline\R:\,\lim_{r\to0}\frac{\mu(B(x,r)\cap\{u>t\})}{\mu(B(x,r))}=0\right\}.
\]
Outside the jump set, i.e. in $X\setminus S_u$, $\mathcal H$-almost every point is a Lebesgue point of $u$ \cite[Theorem 3.5]{KKST}, and we denote the Lebesgue limit at $x$ by $\widetilde{u}(x)$.

The following decomposition result holds for the variation measure of a $\BV$ function. Given an open set $\Omega^*\subset X$, a function $u\in\BV(\Omega^*)$, and a Borel set $A\subset \Omega^*$ that is $\sigma$-finite with respect to $\mathcal H$, we have
\begin{equation}\label{eq:decomposition}
\Vert Du\Vert(\Omega^*)=\Vert Du\Vert(\Omega^*\setminus A)+\int_A\int_{u^{\wedge}(x)}^{u^{\vee}(x)}\theta_{\{u>t\}}(x)\,dt\,d\mathcal H(x),
\end{equation}
see \cite[Theorem 5.3]{AMP}.

A domain $\Omega\subset X$ is said to be $A$-uniform, with constant $A\ge 1$, if for every $x,y\in\Omega$ there exists a curve $\gamma$ in $\Omega$ connecting $x$ and $y$ such that $\ell_{\gamma}\le Ad(x,y)$, and for all $t\in [0,\ell_{\gamma}]$, we have
\[
\dist(\gamma(t),X\setminus\Omega)\ge A^{-1}\min\{t,\ell_{\gamma}-t\}.
\]
We say that a $\mu$-measurable set $\Omega$ satisfies the weak measure density condition if for $\mathcal H$-a.e. $x\in\partial\Omega$,
\begin{equation}\label{eq:weak measure density condition}
\liminf_{r\to 0}\frac{\mu(B(x,r)\cap\Omega)}{\mu(B(x,r))}>0.
\end{equation}
In particular, this is true for any uniform domain \cite{BS}.

\section{The extension result}

In this section we present the extension result for $\BV$ functions.
For any $t>0$ and any set $\Omega\subset X$, we define
\[
\Omega^t:=\{x\in X:\,\dist(x,\Omega)<t\}
\]
and
\[
\Omega_t:=\{x\in \Omega:\,\dist(x,X\setminus \Omega)>t\}.
\]
\begin{theorem}\label{thm:extension theorem}
Let $A\ge 1$, let $\Omega\subset X$ be a bounded $A$-uniform domain, and let $u\in\BV(\Omega)$. Let $T\in (0,\diam(\Omega))$. Then there is an extension $Eu\in\BV(X)$ such that $Eu|_{\Omega}=u$, $\supp(Eu)\subset \Omega^T$,
\[
\Vert Eu\Vert_{\BV(X)}\le C\Vert u\Vert_{\BV(\Omega)},
\]
and $\Vert D(Eu)\Vert(\partial \Omega)=0$. The constant $C$ depends only on the doubling constant, the constants in the Poincar\'e inequality, $A$, and $T$.
\end{theorem}
\begin{remark}\label{rem:uniformity consequences}
According to \cite[Theorem 5.6]{BS}, on which the above result is based, we can in fact replace the uniformity assumption by the following assumptions: $\mu(\partial\Omega)=0$, Lipschitz functions are dense in $N^{1,1}(\Omega)$, $\overline{\Omega}$ satisfies a \emph{corkscrew condition}, and the weighted measure $\dist(x,X\setminus \Omega)^{\alpha}\,d\mu(x)$ with some $\alpha>0$ supports a $(1,1)$-Poincar\'e inequality on $\overline{\Omega}$. For the precise definitions of these concepts, see \cite{BS}. In particular, all of these conditions follow from the uniformity assumption.
\end{remark}
\begin{proof}
In the proof of \cite[Theorem 5.6]{BS}, the authors first construct an extension operator $F:L^1(\Omega)\to L^1(X)$ that maps $\Lip(\Omega)$ to $\Lip(\Omega^T)$. An analysis of the proof reveals that $F$ satisfies for any $t\in (0,T)$
\[
\int_{\Omega^t\setminus\Omega}|Fv|\,d\mu\le C\int_{\Omega\setminus\Omega_{\alpha t}}|v|\,d\mu\quad\textrm{and}\quad\int_{\Omega^t\setminus\Omega}g_{Fv}\,d\mu\le C\int_{\Omega\setminus\Omega_{\alpha t}}g_v\,d\mu,
\] 
where the first inequality holds for any $v\in L^1(\Omega)$, and the second for any $v\in \Lip(\Omega)$, and the constants $C,\alpha>0$ depend only on $c_d$, $c_P$, $\lambda$, $A$, and $T$. Moreover, $g_{Fv}$ and $g_v$ are upper gradients of $Fv$ and $v$. Then one can define a cutoff function
\[
\eta(x):=\max\{0,\min\{1, 2-\dist(x,\Omega)/(T/4)\}\}
\]
and set $Ev:=\eta Fv$. Now, if $v\in\Lip(\Omega)$, then $g_{Ev}:=g_{Fv}\eta+g_{\eta} |Fv|$ is an upper gradient of $Ev$ according to the Leibniz rule (see \cite[Theorem 2.15]{BB}), and we get the estimates for any $t\in (0,T)$:
\begin{equation}\label{eq:Lip estimates L1}
\int_{\Omega^t\setminus\Omega}|Ev|\,d\mu\le C\int_{\Omega\setminus\Omega_{\alpha t}}|v|\,d\mu\qquad\textrm{for }v\in L^1(\Omega),
\end{equation}
and
\begin{equation}\label{eq:Lip estimates gradient}
 \int_{\Omega^t\setminus\Omega}g_{Ev}\,d\mu\le C\int_{\Omega\setminus\Omega_{\alpha t}}g_v\,d\mu+\frac{C}{T} \int_{\Omega\setminus\Omega_{\alpha t}}|v|\,d\mu\qquad\textrm{for }v\in\Lip(\Omega).
\end{equation}
Of course, in the second inequality the number $T$ can then be absorbed into the constant $C$. Now, take a sequence $u_i\in\liploc(\Omega)$ such that $u_i\to u$ in $L_{\loc}^1(\Omega)$ and
\begin{equation}\label{eq:minimizing sequence BV}
\int_{\Omega}g_{u_i}\,d\mu\to \Vert Du\Vert(\Omega)
\end{equation}
as $i\to\infty$. If $u$ is bounded, we can truncate the functions $u_i$, if necessary, to obtain $u_i\to u$ in $L^1(\Omega)$. If $u$ is unbounded, for the truncated functions we have
\[
u^k:=\min\{k,\max\{-k,u\}\}\to u \ \ \textrm{in}\ L^1(\Omega)\ \  \textrm{as}\ k\to\infty,
\]
and $\Vert Du^k\Vert (\Omega)\to \Vert Du\Vert(\Omega)$ by the lower semicontinuity of the variation measure. Thus in any case, we can assume that $u_i\to u$ in $L^1(\Omega)$. As noted in Remark \ref{rem:uniformity consequences}, Lipschitz functions are dense in $N^{1,1}(\Omega)$. Thus we can also assume that $u_i\in \Lip(\Omega)$ for all $i\in\N$. We can extend each function $u_i$ to $Eu_i\in\Lip_c(\Omega^T)$, such that by \eqref{eq:Lip estimates gradient}, we have
\[
\int_{\Omega^T} g_{Eu_i}\,d\mu\le C\int_{\Omega}g_{u_i}\,d\mu+C\int_{\Omega}|u_i|\,d\mu
\]
for every $i\in\N$. Then
\[
\liminf_{i\to\infty}\int_{\Omega^T}g_{Eu_i}\,d\mu\le C\Vert u\Vert_{\BV(\Omega)}.
\]
We can also extend $u$ to $Eu\in L^1(X)$, and by applying \eqref{eq:Lip estimates L1} to $|Eu_i-Eu|$, we get
\[
\int_{X}|Eu_i-Eu|\,d\mu\le C\int_{\Omega}|u_i-u|\,d\mu\to 0\quad\ \textrm{as }i\to\infty.
\]
Thus we have by \eqref{eq:Lip estimates L1} and by the definition of the variation measure,
\[
\int_{X}|Eu|\,d\mu\le C\int_{\Omega}|u|\,d\mu\quad\textrm{and}\quad \Vert D(Eu)\Vert(X)\le C\Vert u\Vert_{\BV(\Omega)}.
\]
This shows that $u\in\BV(X)$ and $\Vert Eu\Vert_{\BV(X)}\le C\Vert u\Vert_{\BV(\Omega)}$, with $C=C(c_d,c_P,\lambda,A,T)$.

Finally let us show that $\Vert D(Eu)\Vert(\partial\Omega)=0$. By \eqref{eq:Lip estimates gradient}, we get for the sequence $u_i\in\Lip(\Omega)$ and for any $t\in (0,T)$ that
\begin{align*}
\limsup_{i\to\infty}\int_{\Omega^t\setminus\overline{\Omega_{\alpha t}}}g_{Eu_i}\,d\mu&\le C\limsup_{i\to\infty}\int_{\Omega\setminus\Omega_{\alpha t}}g_{u_i}\,d\mu+C\limsup_{i\to\infty}\int_{\Omega\setminus\Omega_{\alpha t}}|u_i|\,d\mu\\
&\le C\Vert Du\Vert(\Omega\setminus\Omega_{\alpha t})+C\int_{\Omega\setminus\Omega_{\alpha t}}|u|\,d\mu.
\end{align*}
The last inequality follows from the definition of the variation measure, since we have $u_i\to u$ in $L^1(\Omega)$ and \eqref{eq:minimizing sequence BV}. By using this definition again, and recalling that $Eu_i\to Eu$ in $L^1(X)$, we get
\[
\Vert D(Eu)\Vert(\partial\Omega)\le \Vert D(Eu)\Vert(\Omega^t\setminus\overline{\Omega_{\alpha t}})\le \Vert Du\Vert(\Omega\setminus\Omega_{\alpha t})+C\int_{\Omega\setminus\Omega_{\alpha t}}|u|\,d\mu.
\]
By letting $t\to 0$, we get $\Vert D(Eu)\Vert(\partial\Omega)=0$.
\end{proof}

We give the following definition.

\begin{definition}\label{def:BV extension domain}
An open set $\Omega\subset X$ that satisfies the conclusions of Theorem \ref{thm:extension theorem} is a \emph{strong $\BV$ extension domain}.
Additionally, we say that an open set $\Omega\subset X$ is a $\BV$ extension domain with a constant $c_{\Omega}>0$ if for every $u\in\BV(\Omega)$, there is an extension $Eu\in \BV(X)$ with $Eu|_{\Omega}=u$ and $\Vert Eu\Vert_{\BV(X)}\le c_{\Omega}\Vert u\Vert_{\BV(\Omega)}$.
\end{definition}

Thus the difference between a $\BV$ extension domain and a strong $\BV$ extension domain is that for a $\BV$ extension domain, we do not require $\Vert D(Eu)\Vert(\partial\Omega)=0$.

\section{Traces of $\BV$ functions}

Closely related to extensions is the concept of boundary traces. We give the following definition.

\begin{definition}
For a $\mu$-measurable set $\Omega\subset X$ and a $\mu$-measurable function $u$ on $\Omega$, a function $T_{\Omega}u$ defined on $\partial\Omega$ is a boundary trace of $u$ if for $\mathcal H$-a.e. $x\in\partial\Omega$, we have
\[
\lim_{r\to 0}\,\vint{B(x,r)\cap\Omega}|u-T_{\Omega}u(x)|\,d\mu=0.
\]
\end{definition}

For classical results on boundary traces of $\BV$ functions in the Euclidean setting, see e.g. \cite[Chapter 3]{AFP} or \cite[Chapter 2]{Gi}. As regards the metric setting, in \cite[Theorem 5.7]{HKLL} it was shown that if $\Omega$ is a strong $\BV$ extension domain and satisfies the weak measure density condition \eqref{eq:weak measure density condition}, then the boundary trace $T_{\Omega}u$ exists, that is, $T_{\Omega}u(x)$ is well-defined for $\mathcal H$-a.e. $x\in \partial\Omega$. The proof is based on the fact that since $\Vert D(Eu)\Vert(\partial\Omega)=0$, the boundary $\partial\Omega$ and the jump set $S_{Eu}$ of the extension $Eu$ intersect only in a set of $\mathcal H$-measure zero, due to \eqref{eq:decomposition}.

Next we consider a somewhat different approach to traces, which requires less to be assumed of the set $\Omega$. The following lemma will be useful, and the proof will be similar to the one used in the Euclidean case in \cite{EvGa}.

\begin{lemma}\label{lem:trace type result}
Let $u\in\BV(X)$ and let $\Omega\subset X$ be a $\mu$-measurable set. Consider points $x\in S_u$ for which
\begin{equation}\label{eq:density of Omega}
\liminf_{r\to 0}\frac{\mu(B(x,r)\cap\Omega)}{\mu(B(x,r))}\ge c,
\end{equation}
where $c>0$ is a constant, and
\begin{equation}\label{eq:density estimate for level sets}
\lim_{r\to 0}\,\frac{\mu(B(x,r)\cap \Omega\setminus \{u>t\})}{\mu(B(x,r)\cap\Omega)}=0
\end{equation}
for every $t<u^{\vee}(x)$.
For $\mathcal H$-a.e. such point, we have
\[
\lim_{r\to 0}\,\vint{B(x,r)\cap \Omega}|u-u^{\vee}(x)|^{Q/(Q-1)}\,d\mu=0,
\]
where the number $Q>1$ was defined in \eqref{eq:definition of Q}. The corresponding result for $u^{\wedge}(x)$ can be formulated similarly.
\end{lemma}
\begin{proof}
Let $\eps>0$, and let $x\in X$ satisfy all the assumptions of the lemma. We can also assume that $-\infty<u^{\wedge}(x)<u^{\vee}(x)<\infty$, as this holds for $\mathcal H$-a.e. $x\in S_u$ \cite[Lemma 3.2]{KKST}.
Given any $r>0$, we calculate
\begin{equation}\label{eq:trace type first estimate}
\begin{split}
&\vint{B(x,r)\cap \Omega}|u-u^{\vee}(x)|^{Q/(Q-1)}\,d\mu\\
&\qquad  = \frac{1}{\mu(B(x,r)\cap\Omega)}\int_{B(x,r)\cap\Omega\cap \{u^{\vee}(x)-\eps< u< u^{\vee}(x)+\eps\}}|u-u^{\vee}(x)|^{Q/(Q-1)}\,d\mu\\
&\qquad\qquad  +\frac{1}{\mu(B(x,r)\cap\Omega)}\int_{B(x,r)\cap \Omega \setminus \{u>u^{\vee}(x)-\eps\}}|u-u^{\vee}(x)|^{Q/(Q-1)}\,d\mu\\
&\qquad\qquad  +\frac{1}{\mu(B(x,r)\cap\Omega)}\int_{B(x,r)\cap \Omega \cap \{u\ge u^{\vee}(x)+\eps\}}|u-u^{\vee}(x)|^{Q/(Q-1)}\,d\mu.
\end{split}
\end{equation}
The first term on the right-hand side is clearly at most $\eps^{Q/(Q-1)}$. Let $M>0$ with $-M<u^{\vee}(x)-\eps$. The second term can be estimated as follows:
\begin{align*}
&\frac{1}{\mu(B(x,r)\cap\Omega)}\int_{B(x,r)\cap \Omega \setminus \{u>u^{\vee}(x)-\eps\}}|u-u^{\vee}(x)|^{Q/(Q-1)}\,d\mu\\
&\qquad\quad \le |-M-u^{\vee}(x)|^{Q/(Q-1)}\frac{\mu(B(x,r)\cap \Omega \setminus \{u>u^{\vee}(x)-\eps\})}{\mu(B(x,r)\cap\Omega)}\\
&\qquad\qquad\quad +\frac{1}{\mu(B(x,r)\cap\Omega)}\int_{B(x,r)\cap\{u<-M\}}|u-u^{\vee}(x)|^{Q/(Q-1)}\,d\mu.
\end{align*}
By \eqref{eq:density estimate for level sets}, the first term on the right-hand side goes to zero as $r\to 0$.
The third term of \eqref{eq:trace type first estimate} can be estimated similarly, using the definition of the approximate upper limit, provided we also require that $M>u^{\vee}(x)+\eps$. In total, we get
\begin{align*}
&\limsup_{r\to 0}\,\vint{B(x,r)\cap \Omega}|u-u^{\vee}(x)|^{Q/(Q-1)}\,d\mu\\
&\quad\le\eps^{Q/(Q-1)}+\limsup_{r\to 0}\frac{1}{\mu(B(x,r)\cap\Omega)}\int_{B(x,r)\cap\{|u|>M\}}|u-u^{\vee}(x)|^{Q/(Q-1)}\,d\mu.
\end{align*}
Here we have by \eqref{eq:density of Omega}
\begin{align*}
&\limsup_{r\to 0}\frac{1}{\mu(B(x,r)\cap\Omega)}\int_{B(x,r)\cap\{u>M\}}|u-u^{\vee}(x)|^{Q/(Q-1)}\,d\mu\\
&\qquad\quad\le 2^{Q/(Q-1)}\limsup_{r\to 0}\frac{1}{c\mu (B(x,r))}\int_{B(x,r)}(u-M)_+^{Q/(Q-1)}\,d\mu\\
&\qquad\qquad +2^{Q/(Q-1)}|M-u^{\vee}(x)|^{Q/(Q-1)}\limsup_{r\to 0}\frac{\mu(B(x,r)\cap\{u>M\})}{c\mu(B(x,r))}\\
&\qquad\qquad\qquad \le C\limsup_{r\to 0}\left(r\frac{\Vert D(u-M)_+\Vert(B(x,r))}{\mu(B(x,r))}\right)^{Q/(Q-1)},
\end{align*}
where the last inequality follows from the fact that $M>u^{\vee}(x)$, as well as \eqref{eq:asymptotic poincare}. Moreover, $C=C(c_d,c_P,\lambda,c)$. An analogous estimate holds for the set $\{u<-M\}$, provided that we also have $-M<u^{\wedge}(x)$, and then in total we get
\begin{align*}
&\limsup_{r\to 0}\left(\,\vint{B(x,r)\cap \Omega}|u-u^{\vee}(x)|^{Q/(Q-1)}\,d\mu\right)^{(Q-1)/Q}\\
&\qquad\le \eps+ C\limsup_{r\to 0}r\frac{\Vert D(u-M)_+\Vert(B(x,r))}{\mu(B(x,r))}\\
&\qquad\qquad\qquad\qquad+C\limsup_{r\to 0}r\frac{\Vert D(u+M)_-\Vert(B(x,r))}{\mu(B(x,r))}.
\end{align*}
Since the number $M$ can be chosen to be arbitrarily large, it is straightforward to show that the right-hand side of the above inequality is smaller than $C\eps$ outside a set of arbitrarily small $\mathcal H$-measure (see e.g. the proof of \cite[Theorem 3.5]{KKST}). Since $\eps>0$ was arbitrary, we have the result.
\end{proof}

Before considering traces, we prove the following result which is in close relation with the concept of traces. The theorem strengthens \cite[Theorem 1.1]{KKST}.
\begin{theorem}\label{thm:jump set}
Let $u\in\BV(X)$. Then for $\mathcal H$-a.e. $x\in S_u$, there exist $t_1,t_2\in (u^{\wedge}(x),u^{\vee}(x))$ such that
\[
\lim_{r\to 0}\,\vint{B(x,r)\cap\{u>t_2\}}|u-u^{\vee}(x)|^{Q/(Q-1)}\,d\mu=0
\]
and
\[
\lim_{r\to 0}\,\vint{B(x,r)\cap\{u< t_1\}}|u-u^{\wedge}(x)|^{Q/(Q-1)}\,d\mu=0.
\]
\end{theorem}
\begin{proof}
In this proof we denote, for brevity, the super-level sets of $u$ by $E_t:=\{u>t\}$, $t\in\R$.
By the coarea formula \eqref{eq:coarea}, there is a countable dense set $T\subset \R$ such that for every $t\in T$, the set $E_t$ is of finite perimeter. Let
\begin{equation}\label{eq:def of N}
N:=\bigcup_{t\in T} \left\{x\in \partial^*E_t:\, \eqref{eq:density of E}\textrm{ does not hold at }x\textrm{ with }E\hookrightarrow E_t\right\}
\end{equation}
and
\begin{equation}\label{eq:def of Ntilde}
\widetilde{N}:=\bigcup_{s,t\in T} \left\{x\in \partial^*(E_{s}\setminus  E_{t}):\, \eqref{eq:density of E}\textrm{ does not hold at }x\textrm{ with }E\hookrightarrow E_{s}\setminus E_{t}\right\}.
\end{equation}
Since the sets $E_s\setminus E_t$, $s,t\in T$, are also of finite perimeter by \cite[Proposition 4.7]{Mir}, we have $\mathcal H(N\cup\widetilde{N})=0$.
From the definitions of the lower and upper approximate limits it follows that whenever $x\in S_u$, it is true that $x\in \partial^*E_t$ for every $t\in (u^{\wedge}(x),u^{\vee}(x))$.
Now, at a point $x\in S_u\setminus (N\cup\widetilde{N})$, the number $t_2$ can be chosen as follows. If
\begin{equation}\label{eq:measure theoretic boundary for two level sets}
\limsup_{r\to 0}\frac{\mu(B(x,r)\cap E_{s}\setminus E_{t})}{\mu(B(x,r))}>0
\end{equation}
for $s,t\in T\cap (u^{\wedge}(x),u^{\vee}(x))$, then we have $x\in\partial^*(E_{s}\setminus E_{t})$, and since $x\notin \widetilde{N}$, we actually have
\[
\liminf_{r\to 0}\frac{\mu(B(x,r)\cap E_{s}\setminus E_{t})}{\mu(B(x,r))}\ge \gamma.
\]
Here $\gamma>0$ is a constant. Thus, if \eqref{eq:measure theoretic boundary for two level sets} holds for all consecutive numbers in an increasing sequence $s<t<\ldots \in T\cap (u^{\wedge}(x),u^{\vee}(x))$, the sequence must be finite, and so we have
\[
\limsup_{r\to 0}\frac{\mu(B(x,r)\cap E_{t_2}\setminus E_{t})}{\mu(B(x,r))}=0
\]
for some $t_2\in (u^{\wedge}(x),u^{\vee}(x))\cap T$ and all $t\in (t_2,u^{\vee}(x))$.

Finally, since the set $T$ is countable, the union of the exceptional sets of Lemma \ref{lem:trace type result}, with $\Omega\hookrightarrow E_s$ and $s\in T$, has $\mathcal H$-measure zero. Thus we can assume that $x$ is outside this set, and then
Lemma \ref{lem:trace type result} gives the result for the set $E_{t_2}=\{u>t_2\}$. The proof for the number $t_1$ and the set $\{u< t_1\}$ is analogous.
\end{proof}

Now we proceed to consider traces. First we present an additional assumption on the space $X$.
Following \cite[Definition 6.1]{AMP}, we say that a space satisfies the \emph{locality condition} if, given any two sets of locally finite perimeter $E_1,E_2\subset X$, we have $\theta_{E_1}(x)=\theta_{E_2}(x)$ for $\mathcal H$-a.e. $x\in \partial^*E_1\cap\partial^*E_2$ --- recall the definition of $\theta_E$ from \eqref{eq:def of theta}. The above could as well be formulated with the additional assumption $E_1\subset E_2$, see the discussion following Definition 5.9 in \cite{HKLL}.

Here we give the following stronger condition.

\begin{definition}\label{def:strong locality}
The space $X$ satisfies the \emph{strong locality condition} if, given any two sets of locally finite perimeter $E_1\subset E_2\subset X$, we have for $\mathcal H$-a.e. $x\in \partial^{*}E_1 \cap \partial^{*} E_2$
\begin{equation}\label{eq:locality condition}
\lim_{r\to 0}\frac{\mu((E_2\setminus E_1)\cap B(x,r))}{\mu(B(x,r))}= 0.
\end{equation}
\end{definition}

Strong locality is indeed stronger than locality, as we will soon see.

Condition \eqref{eq:locality condition} is, in particular, satisfied if for any sets of locally finite perimeter $E_1\subset E_2\subset X$, the sets $E_1$ and $E_2$ have the same density at $\mathcal H$-a.e. $x\in \partial^{*}E_1\cap \partial^{*}E_2$.
Furthermore, the assumption $E_1\subset E_2$ is again essentially unnecessary, as the following lemma demonstrates.
\begin{lemma}
Assume that $X$ satisfies the strong locality condition. Let $E_1,E_2\subset X$ be sets of locally finite perimeter. Then for $\mathcal H$-a.e. $x\in \partial^{*}E_1 \cap \partial^{*} E_2$, we have either
\begin{equation}\label{eq:symmetric difference}
\lim_{r\to 0}\frac{\mu ((E_1 \bigtriangleup E_2) \cap B(x,r))}{\mu (B(x,r))}= 0,
\end{equation}
or the same with the substitution $E_2 \hookrightarrow E_2^{c}$ (complement of $E_2$). Here $\bigtriangleup$ is the symmetric difference.
\end{lemma}
\begin{proof}
First we note that $E_2^c$, $E_1\cap E_2$, and $E_1\cap E_2^c$ are also sets of locally finite perimeter, see \cite[Proposition 4.7]{Mir}. Take a point $x\in \partial^{*}E_1 \cap \partial^{*} E_2$. For any sets $A_1\subset A_2\subset X$, let $N_{A_1,A_2}$ be the set of points $x\in \partial^*A_1\cap \partial^*A_2$ for which
\[
\limsup_{r\to 0}\frac{\mu((A_2\setminus A_1)\cap B(x,r))}{\mu(B(x,r))}>0.
\]
By \eqref{eq:locality condition}, excluding a $\mathcal H$-negligible set we can assume that 
\[
x\notin N_{E_1\cap E_2,E_1}\cup N_{E_1\cap E_2,E_2}\cup N_{E_1\cap E_2^c,E_1}\cup N_{E_1\cap E_2^c,E_2^c}.
\]
Since $x\in \partial^{*}E_1 \cap \partial^{*} E_2$, by the definition of the measure theoretic boundary we have either $x\in \partial^{*}(E_1\cap E_2)$ or $x\in \partial^{*}(E_1\cap E_2^c)$. Assume the former. Using \eqref{eq:locality condition}, we can calculate
\begin{align*}
&\mu((E_1 \bigtriangleup E_2)\cap B(x,r))=\mu((E_1\setminus E_2)\cap B(x,r))+\mu((E_2\setminus E_1)\cap B(x,r))\\
                       &\qquad\qquad=\mu((E_1\setminus (E_1\cap E_2))\cap B(x,r))+\mu((E_2\setminus (E_1\cap E_2))\cap B(x,r))\\
                       &\qquad\qquad=o(\mu(B(x,r)))\qquad\textrm{as }r\to 0.
\end{align*}
Then assume that $x\in \partial^{*}(E_1\cap E_2^{c})$ instead. Now we calculate exactly as above, with $E_2$ replaced by $E_2^{c}$. This gives the result.
\end{proof}

To see that strong locality is stronger than locality, we note that if $E_1,E_2\subset X$ are sets of locally finite perimeter, then at $\mathcal H$-a.e. $x\in \partial^*E_1\cap\partial^*E_2$ where \eqref{eq:symmetric difference} is satisfied, we have $\theta_{E_1}(x)=\theta_{E_2}(x)$ \cite[Proposition 6.2]{AMP}.


Let $\Omega\subset X$ be a set of locally finite perimeter, and let $u\in\BV(\Omega^*)$, with $\Omega^*$ open. If the space satisfies the locality condition, we can present the decomposition \eqref{eq:decomposition} with $A=\partial^*\Omega$ in the simpler form
\begin{equation}\label{eq:better decomposition}
\Vert Du\Vert(\Omega^*)=\Vert Du\Vert(\Omega^*\setminus\partial^*\Omega)+\int_{\Omega^*\cap\partial^*\Omega}(u^{\vee}-u^{\wedge})\theta_{\Omega}\,d\mathcal H,
\end{equation}
see \cite[Lemma 5.10]{HKLL}.

Now, consider a space that does not satisfy the strong locality condition, i.e. there are sets of locally finite perimeter $E_1\subset E_2\subset X$ and a set $A\subset \partial^*E_1 \cap \partial^*E_2$ with $\mathcal H(A)>0$, such that
\begin{equation}\label{eq:strong locality fails}
\limsup_{r\to 0}\frac{\mu((E_2\setminus E_1)\cap B(x,r))}{\mu(B(x,r))}>0
\end{equation}
for all $x\in A$.
Again we note that $E_2\setminus E_1$ is also a set of locally finite perimeter. For every $x\in A$ we have $x\in \partial^*(E_2\setminus E_1)$, due to \eqref{eq:strong locality fails}. By \eqref{eq:density of E} we then know that for $\mathcal H$-a.e. $x\in A$, the lower density of $E_2\setminus E_1$ is at least $\gamma>0$. Likewise, the lower densities of $E_2^c$ and $E_1$ are at least $\gamma$ for $\mathcal H$-a.e. $x\in A$. Since these three sets are pairwise disjoint, we must have $3\gamma\le 1$.

Thus we have the following example.

\begin{example}
Let $X$ be a space where $\gamma>1/3$, where $\gamma$ is given in \eqref{eq:density of E}. By the above reasoning, $X$ satisfies the strong locality condition. In particular, in (unweighted) Euclidean spaces, the density of a set of locally finite perimeter is known to be exactly $1/2$ at $\mathcal H$-almost every point of its measure theoretic boundary (see e.g. \cite[Theorem 3.59]{AFP}), so the strong locality condition is satisfied.
\end{example}

By introducing weights, we get further examples.

\begin{example}
Let $(X,d,\mu)$ be a space that satisfies the strong locality condition, and replace the measure $\mu$ with the weighted measure $w\,d\mu$, where the weight $w$ is a nonnegative $\mu$-measurable function that is locally bounded and locally bounded away from zero. From the definition of the perimeter measure it follows that the sets of locally finite perimeter in this space are the same as in the unweighted space. Then it easily follows that this weighted space satisfies the strong locality condition as well. In particular, any weighted Euclidean space equipped with the Euclidean distance and a weighted Lebesgue measure, with the weight locally bounded and locally bounded away from zero, satisfies the strong locality condition.
\end{example}


To begin our analysis of traces, in the following theorem we prove the existence of \emph{interior traces}, which are defined similarly to boundary traces.

\begin{theorem}\label{thm:interior traces}
Assume that $X$ satisfies the strong locality condition. Let $\Omega^*$ be an open set, let $u\in \BV(\Omega^*)$, and let $\Omega$ be a set of locally finite perimeter in $\Omega^*$. Then for $\mathcal H$-a.e. $x\in \Omega^*\cap\partial^{*}\Omega$, we can define the interior traces $\{T_{\Omega}u(x),\,T_{X\setminus \Omega} u(x)\}=\{u^{\wedge}(x),u^{\vee}(x)\}$, which satisfy
\[
\lim_{r\to 0}\,\vint{B(x,r)\cap \Omega}|u-T_{\Omega}u(x)|^{Q/(Q-1)}\,d\mu=0
\]
and
\[
\lim_{r\to 0}\,\vint{B(x,r)\setminus\Omega}|u-T_{X\setminus \Omega}u(x)|^{Q/(Q-1)}\,d\mu=0.
\]
\end{theorem}
\begin{proof}
We know that $\mathcal H$-a.e. $x\in \Omega^*\setminus S_u$ is a Lebesgue point of $u$ with
\[
\lim_{r\to 0}\,\vint{B(x,r)}|u-\widetilde{u}(x)|^{Q/(Q-1)}\,d\mu=0,
\]
see \cite[Theorem 3.5]{KKST}. Moreover, $\mathcal H$-a.e. $x\in\Omega^*\cap\partial^*\Omega$ satisfies \eqref{eq:density of E} with $E\hookrightarrow \Omega$, so in these points we can define both $T_{\Omega} u(x)$ and $T_{X\setminus \Omega} u(x)$ simply as the Lebesgue limit $\widetilde{u}(x)$. Let us then consider $x\in \partial^{*}\Omega\cap S_u$, and again we can assume that \eqref{eq:density of E} is satisfied at $x$ with $E\hookrightarrow \Omega$.
We know that $x\in \partial^{*} \{u>t\}$ for every $t\in (u^{\wedge}(x),u^{\vee}(x))$. Let $T$ be a countable dense subset of $\R$ such that $\{u>t\}$ is of finite perimeter in $\Omega^*$ for every $t\in T$.
By the strong locality condition \eqref{eq:symmetric difference}, we now have for $\mathcal H$-a.e. $x\in \partial^{*}\Omega\cap S_u$ either
\begin{equation}\label{eq:symmetric difference Omega}
\lim_{r\to 0}\frac{\mu((\{u>t\}\bigtriangleup \Omega) \cap B(x,r))}{\mu (B(x,r))}= 0
\end{equation}
for every $t\in(u^{\wedge}(x),u^{\vee}(x))\cap T$, or the same with the substitution $\Omega \hookrightarrow \Omega^{c}$. Note that the fact that $x\in \partial^*\Omega$ rules out the possibility that we could have \eqref{eq:symmetric difference Omega} for some values of $t$, and \eqref{eq:symmetric difference Omega} with $\Omega\hookrightarrow\Omega^c$ for other values of $t$. Assuming \eqref{eq:symmetric difference Omega}, it clearly holds for \emph{every} $t\in (u^{\wedge}(x),u^{\vee}(x))$. Then by Lemma \ref{lem:trace type result} we conclude that for $\mathcal H$-a.e. $x\in \partial^*\Omega\cap S_u$ we can define $T_{\Omega}u(x):=u^{\vee}(x)$, and similarly we get $T_{X\setminus \Omega}u(x)=u^{\wedge}(x)$ --- if \eqref{eq:symmetric difference Omega} holds with $\Omega^{c}$ instead of $\Omega$, these are the other way around.
\end{proof}

Next we show that in a space that satisfies the strong locality condition, Theorem \ref{thm:jump set} can be presented in a simpler form.

\begin{theorem}\label{thm:halfspaces}
Assume that $X$ satisfies the strong locality condition. Let $u\in \BV(X)$. Then for $\mathcal H$-a.e. $x\in S_u$, we have for any $t\in (u^{\wedge}(x),u^{\vee}(x))$
\[
\lim_{r\to 0}\,\vint{B(x,r)\cap\{u>t\}}|u-u^{\vee}(x)|^{Q/(Q-1)}\,d\mu=0
\]
and
\[
\lim_{r\to 0}\,\vint{B(x,r)\cap\{u\le t\}}|u-u^{\wedge}(x)|^{Q/(Q-1)}\,d\mu=0.
\]
\end{theorem}
\begin{proof}
As before, let $T$ be a countable dense subset of $\R$ such that $\{u>t\}$ is of finite perimeter for every $t\in T$, and let $N\subset X$ be defined as in \eqref{eq:def of N}.
Again, we know that for every $x\in S_u$ and every $t\in (u^{\wedge}(x),u^{\vee}(x))$, we have $x\in \partial^*\{u>t\}$.
Let $D\subset X$ consist of the points $x\in\partial^*\{u>t\}$ for some $t\in T$, such that either interior trace of $u$ at $x$ does not exist. 
Since the sets $\{u>t\}$, $t\in T$, are of finite perimeter, by Theorem \ref{thm:interior traces} we have $\mathcal H(D)=0$.  Now, if $x\in S_u \setminus (D\cup N)$, by Theorem \ref{thm:interior traces} we have the result for every $t\in (u^{\wedge}(x),u^{\vee}(x))\cap T$. But since $x\notin N$, we have the result for every $t\in (u^{\wedge}(x),u^{\vee}(x))$.
\end{proof}

Theorem \ref{thm:halfspaces}, and to a lesser extent Theorem \ref{thm:jump set}, are analogues of classical results on $\BV$ functions. To wit, on Euclidean spaces the theorems can be formulated with the level sets $\{u>t\}$ and $\{u\le t\}$ replaced by halfspaces (see \cite[Section 4.5.9]{Fed} or \cite[p. 213]{EvGa}), but in the case of a metric space with the strong locality condition, we must use the level sets which do not necessarily even have density $1/2$ at $x$, see \cite[Example 3.3]{KKST}. However, the lower and upper densities of these sets are restricted by \eqref{eq:density of E}.

Having established the existence of interior traces, we proceed to construct boundary traces. However, for our construction to work, we will again need to assume an extension property. First we present two propositions that are similar to \cite[Proposition 5.11]{HKLL} and \cite[Proposition 5.12]{HKLL}, and are originally based on \cite[Theorem 3.84]{AFP} and \cite[Theorem 3.86]{AFP}.
\begin{proposition}\label{prop:gluing}
Assume that $X$ satisfies the strong locality condition. Let $\Omega^*$ be an open set, and let $\Omega$ be a $\mu$-measurable set with $P(\Omega,\Omega^*)<\infty$. Let $u,v\in \BV(\Omega^*)$ and $w:=u\chi_{\Omega^*\cap\Omega}+v\chi_{\Omega^*\setminus \Omega}$. Then $w\in \BV(\Omega^*)$ if and only if
\begin{equation}\label{eq:trace integrability}
\int_{\Omega^*\cap \partial^* \Omega}|T_{\Omega}u-T_{X\setminus \Omega}v|\,d\mathcal H<\infty.
\end{equation}
Furthermore, we then have
\[
\Vert Dw\Vert(\Omega^*)= \Vert Du\Vert(\Omega^*\cap I)+\Vert Dv\Vert(\Omega^*\cap O)+\int_{\Omega^*\cap\partial^* \Omega}|T_{\Omega}u-T_{X\setminus \Omega}v|\theta_{\Omega}\,d\mathcal H,
\]
where $I$ and $O$ are the measure theoretic interior and exterior (points of density one and zero) of $\Omega$.
\end{proposition}
\begin{proof}
Instead of giving the whole proof, we refer to \cite[Proposition 5.11]{HKLL}. Crucial in the proof is the existence of interior traces on $\partial^*\Omega$,
which is guaranteed by Theorem \ref{thm:interior traces} and the finite perimeter of $\Omega$ in $\Omega^*$. To obtain the final equality, we also need \eqref{eq:better decomposition}. Additionally, in the proof we need the fact that if $w\in\BV(\Omega^*)$, then
\[
\Vert Dw\Vert(\Omega^*\cap I)=\Vert Du\Vert(\Omega^*\cap I)\quad\ \textrm{and}\quad\ \Vert Dw\Vert(\Omega^*\cap O)=\Vert Dv\Vert(\Omega^*\cap O).
\]
The above can be proved with the help of the coarea formula \eqref{eq:coarea} as follows. We have
\begin{equation}\label{eq:Du is Dw}
\begin{split}
\Vert Dw\Vert(\Omega^*\cap I) &\le \Vert Du\Vert(\Omega^*\cap I)+\Vert D(w-u)\Vert(\Omega^*\cap I)\\
&= \Vert Du\Vert(\Omega^*\cap I)+\int_{-\infty}^{\infty}P(\{w-u>t\},\Omega^*\cap I)\,dt.
\end{split}
\end{equation}
Here we have $w-u=0$ $\mu$-almost everywhere in $\Omega^*\cap I$. The inequality opposite to \eqref{eq:Du is Dw} is obtained similarly, so we only need to prove that $P(\{w-u>t\},\Omega^*\cap I)=0$ for a.e. $t\in \R$.
Consider $x\in\Omega^*\cap I$ and $t\in \R$. We have
\begin{align*}
\lim_{r\to 0}&\frac{\mu(B(x,r)\cap \{w-u>t\})}{\mu(B(x,r))}\\
&=\lim_{r\to 0}\frac{\mu(B(x,r)\cap \{w-u>t\}\cap I)}{\mu(B(x,r))}+\lim_{r\to 0}\frac{\mu(B(x,r)\cap \{w-u>t\}\setminus I)}{\mu(B(x,r))}\\
&=\lim_{r\to 0}\frac{\mu(B(x,r)\cap \{w-u>t\}\cap I)}{\mu(B(x,r)\cap I)}+0\\ 
&=\begin{cases}
                0\quad & \textrm{if }t\ge 0\\
                1\quad & \textrm{if }t<0.
\end{cases}
\end{align*}
Thus we have $\partial^*\{w-u>t\}\cap\Omega^*\cap I=\emptyset$ for all $t\in\R$, and it follows that $P(\{w-u>t\},\Omega^*\cap I)=0$ for a.e. $t\in\R$ by \eqref{eq:def of theta}.
\end{proof}

The following proposition on the integrability of traces, or more precisely integrability of lower and upper approximate limits, can be taken directly from \cite[Proposition 5.12]{HKLL}, where a proof is also given.

\begin{proposition}\label{prop:codimension one boundary}
Let $\Omega^*\subset X$ be open, let $u\in \BV(\Omega^*)$, and let $A\subset \Omega^*$ be a bounded Borel set that satisfies $\dist(A,X\setminus \Omega^*)>0$ and
\begin{equation}\label{eq:codimension one condition}
\mathcal H(A\cap B(x,r))\le c_A\frac{\mu(B(x,r))}{r}
\end{equation}
for every $x\in A$ and $r\in (0,R]$, where $R\in(0,\dist(A,X\setminus \Omega^*))$ and $c_A>0$ are constants. Then
\begin{equation}\label{eq:summability of traces}
\int_{A}(|u^{\wedge}|+|u^{\vee}|)\,d\mathcal{H} \le C\Vert u\Vert_{\BV(\Omega^*)},
\end{equation}
where $C=C(c_d,c_P,\lambda,A,R,c_A)$.
\end{proposition}

Now we get the following boundary trace theorem. Here we need an extension property as given in Definition \ref{def:BV extension domain}, but we do not need to require that $\Vert D(Eu)\Vert(\partial\Omega)=0$.
\begin{theorem}\label{thm:boundary traces}
Assume that $X$ satisfies the strong locality condition. Let $\Omega$ be a $\BV$ extension domain with constant $c_{\Omega}>0$, as well as a set of finite perimeter, and let $u\in \BV(\Omega)$. Then for $\mathcal H$-a.e. $x\in\partial^* \Omega$ we can define the boundary trace $T_{\Omega}u(x)$ that satisfies
\[
\lim_{r\to 0}\,\vint{B(x,r)\cap\Omega}|u-T_{\Omega}u(x)|^{Q/(Q-1)}\,d\mu=0.
\]
Moreover, if the assumptions of Proposition \ref{prop:codimension one boundary} are satisfied with $\Omega^*=X$, $A=\partial^* \Omega$, and constants $R,c_{\partial^*\Omega}>0$, then we have
\[
\Vert T_{\Omega}u \Vert_{L^1(\partial^*\Omega,\mathcal H)}\le C\Vert u\Vert_{\BV(\Omega)},
\]
and furthermore $u\chi_{\Omega}\in \BV(X)$ with $\Vert u\chi_{\Omega}\Vert_{\BV(X)}\le C\Vert u\Vert_{\BV(\Omega)}$ (here $u\chi_{\Omega}$ naturally just means the zero extension of $u$ to the whole space $X$). The constant $C$ depends only on $c_d$, $c_P$, $\lambda$, $\Omega$, $R$, $c_{\partial^*\Omega}$, and $c_{\Omega}$.
\end{theorem}
\begin{proof}
Extend $u$ to $Eu\in\BV(X)$. According to Theorem \ref{thm:interior traces}, for $\mathcal H$-a.e. $x\in \partial^{*} \Omega$ there exists an interior trace $T_{\Omega}(Eu)(x)$ that satisfies
\begin{align*}
&\lim_{r\to 0}\,\vint{B(x,r)\cap \Omega}|u-T_{\Omega}(Eu)(x)|^{Q/(Q-1)}\,d\mu\\
&\qquad\quad =\lim_{r\to 0}\,\vint{B(x,r)\cap \Omega}|Eu-T_{\Omega}(Eu)(x)|^{Q/(Q-1)}\,d\mu=0,
\end{align*}
and so we can define the boundary trace $T_{\Omega}u(x)$ simply as $T_{\Omega}(Eu)(x)$ wherever the latter is defined.

To prove the estimate $\Vert T_{\Omega}u \Vert_{L^1(\partial^*\Omega,\mathcal H)}\le C\Vert u\Vert_{\BV(\Omega)}$, we note that for $\mathcal H$-a.e. $x\in \partial^* \Omega$ we have $T_{\Omega}u(x)=T_{\Omega}Eu(x)\in \{(Eu)^{\wedge}(x),(Eu)^{\vee}(x)\}$ by Theorem \ref{thm:interior traces}. By Proposition \ref{prop:codimension one boundary} and the definition of an extension domain, we get
\begin{equation}\label{eq:trace estimate}
\begin{split}
\Vert T_{\Omega}u \Vert_{L^1(\partial^*\Omega,\mathcal H)}\le \Vert |(Eu)^{\wedge}|+|(Eu)^{\vee}|\Vert_{L^1(\partial^*\Omega,\mathcal H)}&\le C\Vert Eu\Vert_{\BV(X)}\\
&\le Cc_{\Omega}\Vert u\Vert_{\BV(\Omega)},
\end{split}
\end{equation}
with $C=C(c_d,c_P,\lambda,\Omega,R,c_{\partial^*\Omega})$.
Finally, by using the fact that $T_{\Omega}Eu=T_{\Omega}u$ as well as \eqref{eq:trace estimate}, we get
\[
\int_{\partial^* \Omega}|T_{\Omega}(Eu)-T_{X\setminus\Omega}0|\,d\mathcal H= \int_{\partial^* \Omega}|T_{\Omega}u|\,d\mathcal H\le C\Vert u\Vert_{\BV(\Omega)}<\infty.
\]
By Proposition \ref{prop:gluing}, this implies that
\[
u\chi_{\Omega}=(Eu)\chi_{\Omega}+0\chi_{X\setminus \Omega}\in \BV(X),
\]
with
\[
\Vert D(u\chi_{\Omega})\Vert(X)=\Vert Du\Vert(\Omega)+\int_{\partial^* \Omega}|T_{\Omega}u|\theta_{\Omega}\,d\mathcal H \le C\Vert u\Vert_{\BV(\Omega)},
\]
where the inequality follows from \eqref{eq:trace estimate}.
\end{proof}

Again, we could give essentially the same result and proof with the assumptions used in \cite{HKLL} --- instead of assuming strong locality and that $\Omega$ is a $\BV$ extension domain, we could assume that $\Omega$ is a strong $\BV$ extension domain, and that both $\Omega$ and its complement satisfy the weak measure density condition \eqref{eq:weak measure density condition}. With either set of assumptions, we conclude that any $u\in\BV(\Omega)$ can be extended to the whole space simply by zero extension. This is, of course, in stark contrast with many other classes of functions, such as Newtonian functions. We also see that while the extendability of a $\BV$ function can be be used to prove the existence of boundary traces, the integrability of the boundary trace $T_{\Omega}u$ on the boundary $\partial\Omega$ with respect to the measure $\mathcal H$ enables, in turn, the function $u$ to be extended by zero extension. This demonstrates the interrelatedness of $\BV$ extensions and boundary traces.\\

\noindent \textbf{Acknowledgements.} The author thanks Juha Kinnunen and Riikka Korte for their feedback on the earlier drafts of the article.
The author was supported by the Finnish Academy of Science and Letters, the Vilho, Yrj\"o and
Kalle V\"ais\"al\"a Foundation.

\noindent Address:

\noindent Aalto University, School of Science, Department of Mathematics and Systems Analysis, P.O. Box 11100, FI-00076 Aalto, Finland. \\
\noindent E-mail: {\tt panu.lahti@aalto.fi}\\

\end{document}